\documentclass[11pt]{article}

\usepackage[margin=1in]{geometry}         

\usepackage{longtable}
\usepackage{graphicx}
\usepackage{caption,subcaption}
\usepackage{stmaryrd}
\usepackage{framed}
\usepackage{amssymb}
\usepackage{epstopdf}
\usepackage{amsmath,amsfonts,amssymb}
\usepackage{enumerate}
\usepackage{multirow}
\usepackage{fixltx2e}
\usepackage{bbm}
\usepackage{url}
\usepackage{cite}
\usepackage{fullpage}
\usepackage{fancyhdr}
\usepackage{empheq}
\usepackage{color}
\usepackage{float}
\usepackage{soul}
\usepackage{graphicx}
\usepackage[doc]{optional}
\usepackage{enumitem}
\usepackage{xcolor}
\usepackage[toc,page]{appendix}
\usepackage{theorem}
\usepackage{array}
\usepackage{booktabs}
\usepackage{epsfig}
\usepackage{latexsym}

\definecolor{labelkey}{rgb}{0,0.08,0.45}
\definecolor{refkey}{rgb}{0,0.6,0.0}

\definecolor{myblue}{rgb}{.9, .9, 1}

\newtheorem{theorem}{Theorem}[section]
\newtheorem{lemma}[theorem]{Lemma}
\newtheorem{corollary}[theorem]{Corollary}
\newtheorem{proposition}[theorem]{Proposition}
\newtheorem{definition}[theorem]{Definition}

\theoremstyle{plain}{\theorembodyfont{\rmfamily}

\theoremstyle{plain}{\theorembodyfont{\rmfamily}
}
\theoremstyle{plain}{\theorembodyfont{\rmfamily}
}
\theoremstyle{plain}{\theorembodyfont{\rmfamily}
}
\theoremstyle{plain}{\theorembodyfont{\rmfamily}
\newtheorem{example}[theorem]{Example}}

\theoremstyle{plain}{\theorembodyfont{\rmfamily}
\newtheorem{remark}[theorem]{Remark}}

\theoremstyle{plain}{\theorembodyfont{\rmfamily}
\newtheorem{customC}{Algorithm}
\newenvironment{Calg}[1]
  {\renewcommand\thecustomC{#1}\customC}
  {\endcustomC}

\theoremstyle{plain}{\theorembodyfont{\rmfamily}

\theoremstyle{plain}{\theorembodyfont{\rmfamily}

\def\proof{\noindent{\it Proof}. \ignorespaces}
\def\endproof{\ensuremath{\hfill \quad \blacksquare}}

\newcommand{\dom}{\ensuremath{\operatorname{dom}}}

\def\RR{{\mathbb{R}}}

\def\NN{{\mathbb{N}}}

\newcommand{\la}{\langle}
\newcommand{\ra}{\rangle}

\newcommand{\nexto}{\kern -0.54em}

\newcommand{\dZ}{{\cal Z \kern -0.7em Z}}
\newcommand{\dC}{{\rm\hbox{C \kern-0.8em\raise0.2ex\hbox{\vrule height5.4pt width0.7pt}}}}
\newcommand{\dQ}{{\rm\hbox{Q \kern-0.85em\raise0.25ex\hbox{\vrule height5.4pt width0.7pt}}}}

\newenvironment{retraitsimple}{\begin{list}{--~}{
 \topsep=0.3ex \itemsep=0.3ex \labelsep=0em \parsep=0em
 \listparindent=1em \itemindent=0em
 \settowidth{\labelwidth}{--~} \leftmargin=\labelwidth
}}{\end{list}}

\begin{document}
\title{Two algorithms for solving systems of inclusion problems}
\author{R. D\'iaz Mill\'an\footnote{ Federal Institute of Goi\'as, Goi\^ania, Brazil, e-mail: rdiazmillan@gmail.com}}
\maketitle
\begin{abstract}
The goal of this paper is to present two algorithms for solving systems of inclusion problems, with all components of the systems being a sum of two maximal monotone operators.
The algorithms are variants of the forward-backward splitting method and one being a hybrid with the alternating projection method. They consist of approximating the
 solution sets involved in the problem by separating halfspaces which is a well-studied strategy. The schemes contain two parts, the first one is an explicit Armijo-type search
 in the spirit of the extragradient-like methods for variational inequalities. The second part is the projection step, this being the main difference between the algorithms.
 While the first algorithm computes the projection onto the intersection of the separating halfspaces, the second chooses one component of the system and projects onto
 the separating halfspace of this case. In the iterative process, the forward-backward operator is computed once per inclusion problem, representing a relevant computational
saving if compared with similar algorithms in the literature. The convergence analysis of the proposed methods is given assuming monotonicity of all operators, without Lipschitz
 continuity assumption. We also present some numerical experiments.


\end{abstract}
\section{Introduction}
 Given a finite family of pairs of operators $\{ A_i,B_i \}_{i \in \mathbb{I}}$, with $\mathbb{I}=:\{1,2,\cdots, m\}$ and $m\in \NN$, the system of inclusion problems consists in:
\begin{equation}\label{problema}
\mbox{find} \ \ x^*\in \RR^n  \ \  \mbox{such that} \ \  0\in A_i(x^*)+B_i(x^*) \ \ \mbox{for all} \ \ i\in \mathbb{I},
\end{equation}
where the operators $A_i:\dom(A_i)\subset \RR^n \rightarrow \RR^n$ are point-to-point and maximal monotone and the operator $B_i:\dom(B_i)\subset \RR^n\rightarrow 2^{\RR^n}$
is point-to-set maximal monotone, for all $i\in\mathbb{I}$. The solution set of the problem, denoted by $S_*$, is given by the intersection of the solution sets of all components of the system
, i.e., $S_*=\cap_{i\in \mathbb{I}} S^i_*$, where $S^i_*$ is defined as $S^i_*:=\{x\in \RR^n: 0\in A_i(x)+B_i(x)\}$.

 Many problems in mathematics can be modeled as problem \eqref{problema}, for example, taking the operators $B_i = N_{C_i}$ (the normal cone
 of $C_i$) with $C_i \subseteq \RR^n$ nonempty, closed and convex set for all $i\in \mathbb{I}$, we obtain a system of variational inequalities, introduced by I.V. Konnov in \cite{konnov1}, which
have been deeply studied, see \cite{konnov, eslam, gibali, gibali2, gibali3, konnov1,van, alm}.

  For solving inclusion problems for the sum of two operators, the hypothesis of Lipschitz continuity and the forward-backward method has been used regularly, see \cite{tseng, ranf, silvia, silvia2}. Due to its extensive field of applications, it is crucial to consider general versions of problem \eqref{problema} which relax the Lipschitz continuity hypothesis. That is the reason we are interested in assuming only maximal monotonicity for all operators involved, without the Lipschitz continuity assumption.

 The proposed algorithms contain two main steps, a line-search for finding a separating hyperplane, and a projection onto separating hyperplanes. The line-search compute the operator forward-backward once per each component of the system at each iteration, which represents a relevant computational saving in comparison with the line-search proposed by Tseng in \cite{tseng}. The second part consists in projecting the current point onto a suitable set. This part is the main difference between our algorithms.
 In the first algorithm, we project onto the intersection of the separating hyperplanes. While the second algorithm chooses one component of the system, find the separating hyperplane and project onto it. This second method is a hybrid with the alternating projection method.

 In the first algorithm, we calculate in parallel, the hyperplane separating the current point and the solution set of each component of the system, and project onto the intersection
 of all of them. In the second one, we only use one component of the system in each step of the algorithm,
in the spirits of the alternating projection method. The present work follows the ideas presented in \cite{rei-yun, phdthesis, borw-baus}.

  The number of hyperplanes that must be intersected for the first algorithm is at most equal to the number of components of the system, contrary to the algorithms proposed in \cite{van}, in which the number of hyperplanes increases at each iteration. This makes our scheme be computationally lightest. The numerical experiments (see Section \ref{examples}) shows that the search of the separating hyperplanes is more efficient (computationally) than compute the projection onto the intersection of many hyperplanes.

The problem \eqref{problema} has many applications in operations research, optimal control, mathematical physics, optimization and differential equations.
This kind of problem has been receiving an increasing academic attention in the recent years. It's is due to a fact that many nonlinear problems arise within applied areas, are mathematically modeled as
 nonlinear operator system of equations and/or inclusions, we can refer to  \cite{gibali, van, alm,gibali2, gibali3, semenov, eslam}.

The present work is organized as follow. The next section contains some notations and preliminary results useful for the remainder of this paper. The variants of the forward-backward
 splitting method we present in section \ref{section3}. In section \ref{section4} the convergence analysis of both algorithms is proved. Section \ref{examples} is dedicated to showing
some numerical experiments and comparison with a similar method in the literature and between our algorithms. Finally, we provide some conclusions.

\section{Preliminaries}
 In this section, we review some basic definitions and results. First, we introduce the notation and recall some definitions. The inner product in
$\RR^n$ is denoted by $\la \cdot , \cdot \ra$ and its induced norm by $\|\cdot\|$. We denote by $2^{C}$ the power set of the set $C$, and by $B[0,R]$ the closed ball
 centered in $0$ and radius $R$. Given a nonempty, convex and
closed subset $X$ of $\RR^n$, we denote by $P_X(x)$, the orthogonal projection of $x$
onto $X$. It's defined as the unique point in $X$, such that $\|P_X(x)-x\| \le \|y-x\|$ for all $y\in X$. By $N_X(x)$
we denote the normal cone of $X$ in $x\in X$, defined as
$N_X(x):=\{d\in \RR^n\, : \, \la d,x-y\ra\ge 0 \;\; \forall
y\in X\}$. The {\it domain} of $T$ is defined by, $\dom(T):=\{x\in \RR^n: T(x)\neq \emptyset\}$. The operator $T:\dom(T)\subset \RR^n \rightarrow 2^{\RR^n}$
is said to be monotone if, for all $(x,u),(y,v)$ in the {\it graph} of $T$, ($Gr(T):=\{(x,u)\in
\RR^n\times \RR^n : u\in T(x)\}$), we have $\la x-y, u-v \ra \ge 0,$
and it is maximal if $T$ has no proper monotone extension in the
graph inclusion sense.

We start with the well-known definition of the so-called Fej\'er convergence also know as Fej\'er monotonicity.

\begin{definition}
Let $S$ be a nonempty subset of $\RR^n$. The sequence $(x^k)_{k\in \NN}\subset \RR^n$ is said to be Fej\'er convergent to $S$, if and only if, for all $x\in S$ there exists $k_0\ge 0$, such that $\|x^{k+1}-x\| \le \|x^k - x\|$ for all $k\ge k_0$.
\end{definition}

This definition was introduced in \cite{browder} and have been
further elaborated in \cite{IST,combe} and \cite{borw-baus} and references therein. A useful result on Fej\'er sequences is the following.

\begin{proposition}\label{punto}
If the sequence $(x^k)_{k\in \NN}$ is Fej\'er convergent to $S\neq \emptyset$, then:
\begin{enumerate}
\item  $(x^k)_{k\in \NN}$ is bounded;
\item  $(\|x^k-x\|)_{k\in \NN}$ is convergent for all $x\in S;$
\item if one cluster point $x^*$ of $(x^k)_{k\in\NN}$ belongs to $S$, then the sequence $(x^k)_{k\in \NN}$ converges to $x^*$.
\end{enumerate}
\end{proposition}

\begin{proof}
(i) and (ii) See Proposition $5.4$ in \cite{librobauch}. (iii) See Theorem $5.5$ in \cite{librobauch}.
\end{proof}

We following with some known results on the orthogonal projection that will be useful for the well-definition of the stopping criteria. Moreover, for proving the Fej\'{e}r convergence of the sequence generated by the algorithms.

\begin{proposition}\label{proj}
Let $X$ be any nonempty, closed and convex set in $\RR^n$. For all $x,y\in \RR^n$ and all $z\in X $ the following hold:
\begin{enumerate}
\item $ \|P_{X}(x)-P_{X}(y)\|^2 \leq \|x-y\|^2-\|(P_{X}(x)-x)-\big(P_{X}(y)-y\big)\|^2.$
\item $\la x-P_X(x),z-P_X(x)\ra \leq 0.$
\end{enumerate}
\end{proposition}

\begin{proof}
 (i) and (ii) see    Lemma    $1.1$    and    $1.2$    in    \cite{zarantonelo}.
\end{proof}

Now  we state some useful results on maximal monotone operators. The next proposition will be useful for proving the convergence of the sequences generated by both algorithms.

\begin{proposition}\label{inversa}
Let $T:dom(T)\subseteq\RR^n \rightarrow 2^{\RR^n}$ be a point-to-set and maximal monotone operator. If $\beta >0$ then the operator $(I+\beta\, T)^{-1}: \RR^n \rightarrow dom(T)$
 is single valued and maximal monotone.
\end{proposition}
\begin{proof}
See Theorem $4$ in \cite{minty}.
\end{proof}

The next proposition will be used for the well-definition of the stopping criteria and the convergence of the sequences generated by both algorithms.
\begin{proposition}\label{parada}
Given $\beta>0$ and the maximal monotone operators $A: dom(A)\subseteq \RR^n\to \RR^n$  and $B: dom(B)\subseteq \RR^n\rightarrow 2^{\RR^n}$, if $x\in \dom(A)\cap \dom(B)$ then
 $$x=(I+\beta B)^{-1}(I-\beta A)(x),$$ if and only if, $0\in (A+B)(x)$.
\end{proposition}
\begin{proof}
See Proposition $3.13$ in \cite{PhD-E}.
\end{proof}

Now we prove a lemma which ensures that the hyperplane used in the algorithms contains the solution set of problem \eqref{problema}.

\begin{lemma}\label{propseq}
Given a families of operators $\{A_i,B_i\}_{i\in\mathbb{I}}$, such that for all $i\in \mathbb{I}$ $\dom B_i\subseteq \dom A_i$, take $x,u \in \RR^n$ with $x\in \dom B_i$   for all $i\in \mathbb{I}$. Define:
\begin{equation}\label{H(x)}
H_i(x,u) := \big\{ y\in \RR^n :\la A_i(x)+u,y-x\ra\le 0\big \}.
\end{equation}
  Then for all $(x,u)\in Gr(B_i)$, $S_*^i\subseteq H_i(x,u)$, for all $i\in \mathbb{I}$. Therefore $S_* \subset H_i(x,u)$ for all $i \in \mathbb{I}$.
\end{lemma}

\begin{proof}
Take $x^{*}\in S_*^i$. Using the definition of the solution, there exists $v^{*}\in B_i(x^{*})$, such that $0=A_i(x^{*})+v^{*}$. By the monotonicity of $A_i+B_i$, we have
$$\la A_i(x)+u -(A_i(x^{*})+v^{*}), x-x^{*}\ra\ge 0, $$
 for all $(x,u)\in Gr(B_i)$.
Hence,
$$\la A_i(x)+u, x^{*}-x\ra \le 0$$
and by \eqref{H(x)}, $x^{*}\in H_i(x,u)$.
\end{proof}

\section{The Algorithms}\label{section3}
In this section, we present two algorithms for solving the problem \eqref{problema}. For all $i\in \mathbb{I}$, let  $A_i:\dom(A_i)\subset\RR^n \rightarrow \RR^n$ be point-to-point maximal monotone operators and $B_i:\dom(B_i)\subset\RR^n\rightarrow 2^{\RR^n} $ be point-to-set and maximal monotone operators. We assume that:
\begin{enumerate}[leftmargin=0.5in, label=({\bf A\arabic*})]

\item\label{a1} $dom (B_i)\subseteq dom (A_i)$, for all $i \in \mathbb{I}:=\{1,2,3, \cdots, m\}$ with $m\in \NN$.
\item \label{a2} $S_*\ne \emptyset$.
\item\label{a3} For each bounded and closed subset $V \subset \cap_{i=1}^m dom(B_i)$ there exists $R>0$, such that $B_i(x)\cap B[0,R]\neq\emptyset$,  for all $x\in V$ and $i\in \mathbb{I}$.
\item \label{a4} For all $i\in \mathbb{I}$, the operator $A_i$ is continuous on $\dom (A_i)$.
\end{enumerate}

Assumptions \ref{a1} and \ref{a2} are standard in the literature. We emphasize that assumption \ref{a3} holds trivially if $dom(B_i)=\RR^n$ or $V\subset int(dom(B_i))$ or $B_i$ is the normal cone of any subset of $dom(B_i)$ for all $i\in\mathbb{I}$, i.e.,
 in the application to systems of variational inequality problems, this assumption is trivially satisfied. The operators $A_i$ for $i\in\mathbb{I}$ are all continuous on the interior of its domain by maximality, then the Assumption \ref{a4} is for ensure the continuity on the boundary of the domain.  Note that, when $\dom(A_i)$ is open for all $i\in \mathbb{I}$ this assumption is not required.  Even more, when the set $X$ (defined bellow) is a subset of the interior of $\cap_{i\in\mathbb{I}}$, also the assumption is not necessary.

Choose any nonempty, closed, bounded and convex set, $X \subseteq \cap_{i\in \mathbb{I}}dom (B_i)$, satisfying  $X\cap S_*\ne \emptyset$. There exist some choices for $X$, for example, when the sets $\dom(B_i)$ are closed  for all $ i\in\mathbb{I}$, so they are convex (see \cite{minty2}), we can think in $X=B[0,L]\cap_{i\in \mathbb{I}}dom (B_i)$ for $L$ be large enough, see more details in \cite{tseng, rei-yun, phdthesis}. The necessity of $X$ be bounded is only for the applicability of the Assumption $\ref{a3}$, then in a case of systems of variational inequalities problem, the set $X$ can be unbounded, like was used in \cite{tseng}. For example, consider $B_i=N_{C_i}$ for all $i\in\mathbb{I}$, then $X=\cap_{i\in \mathbb{I}}C_i$ is a good choice. See Section \ref{examples} for specific choices of $X$ in some examples.

For both algorithms we consider the sequence $(\beta_k)_{k=0}^{\infty}$ satisfying that $(\beta_k)_{k\in \NN}\subseteq [\check{\beta},\hat{\beta}] $ for $0<\check{\beta} \leq \hat{\beta}<\infty$, and $\theta, \delta\in(0,1)$, let $R>0$ as in Assumption \ref{a3} taking $V=X$. The algorithms are defined as follows:
\begin{center}
\fbox{\begin{minipage}[b]{\textwidth}
\begin{Calg}{1}\label{concep} Let $(\beta_k)_{k\in \NN}, \theta, \delta, R \mbox{ and } \mathbb{I}$ like above.
\begin{retraitsimple}
\item[] {\bf Step~0 (Initialization):} Take $x^0\in X$.

\item[] {\bf Step~1 (Iterative Step 1):} Given $x^k$, compute for all $i\in \mathbb{I}$,
\begin{equation}{\label{jota1}}
J_i(x^k,\beta_{k}):=(I+\beta_{k}B_i)^{-1}(I-\beta_{k}A_i)(x^{k}).
\end{equation}
\item[]{\bf Step~2 (Stopping Criteria 1):} Define $\mathbb{I}_k^*:=\{i\in \mathbb{I}: x^k=J_{i}(x^k,\beta_k)\}$. If $\mathbb{I}_k^*=\mathbb{I}$ stop.
\item[]{\bf Step~2.5 (Definition):} $\forall i\in \mathbb{I}_k^*$ define $\bar{x}_i^k:=x^k$ and $\bar{u}_i^k:=-A_i(x^k)\in B_i(x^k)$.
\item[] {\bf Step~3 (Inner Loop):} Otherwise, for all $i\in \mathbb{I}\setminus \mathbb{I}_k^*$ begin the inner loop over $j$.
 Put $j=0$ and choose any $u_{(j,i)}^{k}\in B_i\big(\theta^{j}J_i(x^{k},\beta_k)+(1-\theta^{j})x^k\big)\cap B[0,R]$. If
\begin{equation}\label{jk1}
\Big \la A_i \big(\theta^{j}J_i(x^{k},\beta_k)+(1-\theta^{j})x^k\big)+u^{k}_{(j,i)}, x^k-J_i(x^k,\beta_k)\Big \ra\geq \frac{\delta}{\beta_k}\|x^k -J_i(x^k,\beta_k)\|^2,
\end{equation}
then $j_i(k):=j$ and stop.
Else, $j=j+1$.
\item[] {\bf Step~4 (Iterative Step 2):} Set for all $i\in \mathbb{I}\setminus \mathbb{I}_k^*$

\begin{equation}{\label{alphak}}
\alpha_{k,i}:=\theta^{j_i(k)},
\end{equation}
\begin{equation}{\label{ubar}}
\bar{u}_i^k:=u^k_{j_i(k),i}
\end{equation}
\begin{equation}{\label{xbar}}
\bar{x}_i^k:=\alpha_{k,i} J_{i}(x^k,\beta_k)+(1-\alpha_{k,i})x^k
\end{equation}
and
\begin{equation}{\label{Fk}}
x^{k+1}:=P_X\big(P_{H_k}(x^k)\big).
\end{equation}
\item[] {\bf Step~5 (Stopping Criteria 2):} If $x^{k+1}=x^k$ then stop. Otherwise, set $k\leftarrow k+1$ and go to {\bf Step~1}.
\end{retraitsimple}
\end{Calg}\end{minipage}}
\end{center}
where $H_i(x,u)$ as in \eqref{H(x)},
\begin{equation}\label{hk}
H_k:=\cap_{i\in \mathbb{I}\setminus \mathbb{I}_k^*}H_i( \bar{x}_{i}^{k},\bar{u}_i^k).
\end{equation}
Observe that for all $i\in \mathbb{I}_k^*$ using Proposition \ref{parada} we have that $0\in A_i(\bar{x}_i^k)+B_i(\bar{x}_i^k)$, hence, $-A_i(\bar{x}_i^k)\in B_i(\bar{x}_i^k)$, proving that in this case $H_i(\bar{x}_i^k,\bar{u}_i^k)=\RR^n$. By definition of $H_k$ and {\bf Step 2.5} of  {\bf Algorithm \ref{concep}} we have that $H_k\subseteq H_{i}(\bar{x}_i^k,\bar{u}_i^k)$ for all $i \in \mathbb{I}$.

Note that in {\bf Algorithm \ref{concep}}, we project onto the intersection of the separating hyperplanes, which are at most $m$. When the number of component of the system is large, this method requires, at each iteration,  solving a non-trivial subproblem. But nevertheless, in \cite{van}, this intersection is computed onto a largest number of hyperplanes. In view of this possible drawback, we propose the second algorithm, in which we do not need to compute any intersection.

For the second algorithm we will make use of the function $\rho:\NN \rightarrow \mathbb{I}$, that is any surjective and periodic function. We can choose
for example,  the function remainder after division by $m$, $\rho(n)=rem(n,m)$ and defined by $\rho(mk)=m$ for all $k\in \NN$.

 {\bf Algorithm \ref{concep1}}, combines the Alternating Projection Method, the Forward-Backward Method and the ideas of the separating hyperplane.
 Note that in {\bf Algorithm \ref{concep1}}, the iterative process does not depend on the number of equations involved in the system. At each iteration we use only one component
of the system. Hence, this algorithm is recommended for systems with a large number of components. We refer the reader to the papers \cite{reich, reich1} which have similar ideas on Alternating Projection Algorithm.

\begin{center}
\fbox{\begin{minipage}[b]{\textwidth}
\begin{Calg}{2}\label{concep1} Let $(\beta_k)_{k\in \NN}, \theta, \delta, R \mbox{ and } \mathbb{I}$ like above.
\begin{retraitsimple}
\item[] {\bf Step~0 (Initialization):} Take $x^0\in X$.

\item[] {\bf Step~1 (Iterative Step 1):} Given $x^k$, compute:
\begin{equation}{\label{jota}}
J_{\rho(k)}^{k}:=(I+\beta_{k}B_{\rho(k)})^{-1}(I-\beta_{k}A_{\rho(k)})(x^{k}).
\end{equation}
\item[] {\bf Stopping Criteria 1} If $x^k=J_{\rho(k)}^k$ put $\rho(k)\in \mathbb{I}^*_k$ set $k=k+1$ and go to {\bf Step 1}. If $\mathbb{I}_k^*=\mathbb{I}$ , then $x^k \in S_*$.
\item[] {\bf Step~1.1 (Inner Loop):} Begin the inner loop over $j$.
 Put $j=0$ and choose any \newline $u_{(j,\rho(k))}^{k}\in B_{\rho(k)}\big(\theta^{j}J_{\rho(k)}^k+(1-\theta^{j})x^k\big)\cap B[0,R]$. If
\begin{equation}\label{jk}
\Big \la A_{\rho(k)} \big(\theta^{j}J_{\rho(k)}^k+(1-\theta^{j})x^k\big)+u^{k}_{(j,\rho(k))}, x^k-J_{\rho(k)}^k\Big \ra\geq \frac{\delta}{\beta_k}\|x^k -J_{\rho(k)}^k\|^2,
\end{equation}
then $j(k):=j$ and stop.
Else, $j=j+1$.
\item[] {\bf Step~2 (Iterative Step 2):} Define:
\begin{equation}{\label{alphak}}
\alpha_{k}:=\theta^{j(k)},
\end{equation}
\begin{equation}{\label{ubar}}
\bar{u}^k:=u^k_{j(k),\rho(k)}
\end{equation}
\begin{equation}{\label{xbar}}
\bar{x}^k:=\alpha_{k} J_{\rho(k)}^k+(1-\alpha_{k})x^k
\end{equation}
\begin{equation}{\label{Fk}}
x^{k+1}=P_X\big(P_{H_{\rho(k)}(\bar{x}^{k},\bar{u}^{k})}(x^{k})\big).
\end{equation}
set $k=k+1$, empty $\mathbb{I}^*_k$ and go to {\bf Step 1}.
\end{retraitsimple}
\end{Calg}\end{minipage}}
\end{center}

where $H_i(x,u)$  as in \eqref{H(x)}.

\section{Convergence Analysis}\label{section4}
In this section, we analyze the convergence of the algorithms presented in the previous section. First, we present some general properties as well as prove the well-definition of both algorithms.

\noindent From now on, $(x^k)_{k\in \NN}$ is the sequence generated by the algorithm.
\begin{proposition}\label{propdef}
 In both algorithms the {\bf Inner Loop} is well-defined.
\end{proposition}
\begin{proof}
 Here we use $i$ as in the {\bf Algorithm 1}, but nothing change if we use $\rho(k)$ as in {\bf Algorithm 2}, both algorithms have the same {\bf Inner Loop}. The proof of the well-definition of $j_i(k)$ is by contradiction. If {\bf Algorithm 1} or {\bf 2} reaches the {\bf Inner Loop}, then $i \notin \mathbb{I}^*_k$ . Now, assume that for all $j\ge0$ having chosen $u_{(j,i)}^{k}\in B_{i}\big(\theta^j J_{i}^k+(1-\theta^j)x^k\big)\cap B[0,R]$,
\begin{equation*}
\Big\la A_{i} \big(\theta^{j}J_{i}^k+(1-\theta^{j})x^k\big)+u^{k}_{j,i}, x^k-J_{i}^k\Big\ra < \frac{\delta}{\beta_k}\|x^k - J_{i}^k\|^2.
\end{equation*}
Since the sequence $(u^{k}_{(j,{i})})_{j=0}^{\infty}$ is bounded, there exists a subsequence $(u^{k}_{(\ell_j,{i})})_{j=0}^{\infty}$ of $(u^{k}_{(j,{i})})_{j=0}^{\infty}$,
 which converges to an element $ u_{i}^k$ belonging to $B_{i}(x^k)$ by closed graph property, see Proposition 4.2.1(ii) in \cite{librobauch}. Taking the limit over the subsequence $(\ell_j)_{j\in \NN}$, we get
\begin{equation}{\label{lim}}
\big\la\beta_k A_{i}(x^k)+\beta_k u_{i}^k, x^k -J_{i}^k\big \ra \le \delta \|x^k - J_{i}^k\|^2.
\end{equation}
It follows from (\ref{jota1}) that
\begin{equation*}{\label{res}}
 \beta_k A_i(x^k)=x^k-J_i^k-\beta_k v_i^k,
 \end{equation*}
 for some  $ v_i^{k}\in B_i(J_i^k)$.\\
Now, the above equality together with (\ref{lim}), lead to
\begin{equation*}
\|x^k - J_i^k\|^2\le\Big\la x^k-J_i^k-\beta_k v_i^k+\beta_k u_i^k, x^k -J_i^k\Big \ra \le \delta \|x^k - J_i^k\|^2,
\end{equation*}
using the monotonicity of $B_i$ for the first inequality. So,
$$(1-\delta)\|x^k - J_i^k\|^2\le 0,$$
implying that $x^k=J_i^k$, which contradicts that $i\in \mathbb{I}\setminus \mathbb{I}_k^*$. Thus, the algorithm is well-defined.
\end{proof}

A useful algebraic property on the sequence generated by  {\bf Algorithm 1} and {\bf 2}, which is a direct consequence of the {\bf Inner Loop}, is the following.
\begin{corollary}\label{coro}
Let $(x^k)_{k\in \NN}$, $(\beta_k)_{k\in \NN}$ and $(\alpha_{(k,i)})_{k\in \NN}$ be sequences generated by {\bf Algorithm 1} or {\bf 2}. With $\delta$ and $\hat{\beta}$ as defined
 in the algorithms. Then,
\begin{equation}\label{desig-muy-usada}
\la A_i(\bar{x}_i^{k})+\bar{u}_i^{k},x^{k}-\bar{x}_i^{k} \ra  \ge\frac{\alpha_{k,i}\delta}{\hat{\beta}}\|x^{k}-J_i(x^{k},\beta_{k})\|^2\geq 0,
\end{equation}
for all $k$.
\end{corollary}
Note that we have the same property if we replace $i$ by $\rho(k)$.

The following proposition shows that the {\bf Stopping Criteria 1} of both algorithms are well defined.
\begin{proposition}\label{stop1}
If {\bf Algorithm 1} or {\bf 2} stops at iteration $k$ by the {\bf Stopping Criteria 1}, then $x^k\in S_*$.
\end{proposition}

\begin{proof}
If {\bf Stopping Criteria 1} is satisfied, then $\mathbb{I}_k^*=\mathbb{I}$ for both algorithms, then by Proposition \ref{parada} we have that $x^k \in S_*^i$ for all $i\in \mathbb{I}$ which imply that $x^k\in S_*$.
 \end{proof}

\subsection{Convergence Analysis of {\bf  Algorithm \ref{concep}}}
This subsection is dedicated to prove the convergence of the {\bf Algorithm \ref{concep}}.
\begin{proposition}\label{H-separa-x} $x^k \in H_k$  if and only if, $x^k\in S^*$.
\end{proposition}
\begin{proof}
 Direct consequence of Proposition 4.2 of \cite{rei-yun} and the definition of $H_k$.
\end{proof}

Now we prove that the {\bf Stopping Criteria 2} is well defined.
\begin{proposition}\label{stop1}
If {\bf Stop Criteria 2} is satisfied, then $x^k \in S_*$.
\end{proposition}
\proof
If $x^{k+1}=P_X\big(P_{H_k}(x^k)\big)=x^k$, using Proposition \ref{proj}(ii), we have
\begin{equation}\label{proyex}
\la P_{H_k}(x^k)-x^k, z-x^k\ra \leq 0,
\end{equation} for all $z\in X$. Now using Proposition \ref{proj}(ii) again,
\begin{equation}\label{proyeh}
\la P_{H_k}(x^k)-x^k, P_{H_k}(x^k)-z\ra \leq 0,
\end{equation} for all $z\in H_k$.
Since $X\cap H_k \neq \emptyset$ summing \eqref{proyex} and \eqref{proyeh}, with $z\in X\cap H_k$, we get
\begin{equation*}
\|x^k-P_{H_k}(x^k)\|^2=0.
\end{equation*}
Hence, $x^k=P_{H_k}(x^k)$, implying that $x^k\in H_k$ and by Proposition \ref{H-separa-x}, $x^k\in S^*$.
\endproof

From now on assume that {\bf Algorithm 1} generate an infinite sequence $(x^k)_{k\in\NN}$. The next property show some good properties on the sequence generated by {\bf Algorithm 1}.

\begin{proposition}\label{prop21}
\begin{enumerate}
\item The sequence $(x^k)_{k\in \NN}$ is Fej\'er convergent to $S^*\cap X$.
\item The sequence $(x^k)_{k\in \NN}$ is bounded.
\item $\lim_{k\to \infty}\|P_{H_k}(x^k)-x^k\|=0$.
\item $\lim_{k\to \infty}\|x^{x+1}-x^k\|=0$.
\end{enumerate}
\end{proposition}
\begin{proof}
(i) Take $x^*\in S^*\cap X$.  Using  Proposition \ref{proj}(i) and Lemma \ref{propseq}, we have
\begin{eqnarray}\label{fejer-des1}\nonumber\|x^{k+1}-x^{*}\|^2&=&\|P_{X}(P_{H_k}(x^k))-P_{X}(P_{H_k}(x^{*}))\|^2\le \|P_{H_k}(x^k)-P_{H_k}(x^{*})\|^2\\&\leq& \|x^k-x^*\|^2-\|P_{H_k}(x^k)-x^k\|^2.\end{eqnarray} So, $\|x^{k+1}-x^{*}\|\le \|x^k-x^*\|$.

(ii) Follows immediately from item (i) and Proposition \ref{punto}(i).

(iii)Take $x^* \in S^*\cap X$. Using \eqref{fejer-des1} yields
\begin{equation}\label{ineq1}
\|P_{H_k}(x^k)-x^k\|^2\le \|x^k-x^*\|^2-\|x^{k+1}-x^{*}\|^2.
\end{equation}
Now using Proposition \ref{punto}(ii) and item (ii) we have that the right side of equation \eqref{ineq1} goes to zero. Obtaining the result.

(iv) Since the sequence $\{x^k\}_{k\in \NN}$ belongs to $X$, we have
$$\|x^{k+1}-x^k\|^2=\|P_{X}(P_{H_k}(x^k))-P_{X}(x^k)\|^2\le \|P_{H_k}(x^k)-x^k\|^2.$$
Taking limits in the above equation and using the previous item we have the result.
\end{proof}

The following proposition gives us an important behavior of the sequences involved in the algorithm.
\begin{proposition}\label{cadai}
For all $i\in \mathbb{I}$ we have,
$$\lim_{k\to \infty}\la A_i(\bar{x}_i^k)+\bar{u}_i^k,x^k-\bar{x}_i^k \ra=0.$$
\end{proposition}
\begin{proof}
For all $i\in \mathbb{I}$. Using the fact that $H_k \subseteq H_i(\bar{x}_i^k,\bar{u}_i^k)$ by \eqref{hk} and {\bf Step 2.5}, we have that,
$$\|P_{H_i(\bar{x}_i^k,\bar{u}_i^k)}(x^k)-x^k\|^2\leq \|P_{H_k}(x^k)-x^k\|^2.$$

Using the fact that,
$$P_{H_i(\bar{x}_i^k,\bar{u}_i^k)}(x^k)=x^k-\frac{\la A_i(\bar{x}_i^k)+\bar{u}_i^k,x^k-\bar{x}_i^k  \ra}{\|A_i(\bar{x}_i^k)+\bar{u}_i^k\|^2}(A_i(\bar{x}_i^k)+\bar{u}_i^k),$$
and the previous equation, we have
\begin{equation}\label{pasar al lim1}
\frac{\big(\la A_i(\bar{x}_i^k)+\bar{u}_i^k,x^k-\bar{x}_i^k  \ra\big)^2}{\|A_i(\bar{x}_i^k)+\bar{u}_i^k\|^2}\leq \|P_{H_k}(x^k)-x^k\|^2.
\end{equation}
By Proposition \ref{inversa} and the continuity of $A_i$ we have that $J_i$ is continuous, since $(x^k)_{k\in \NN}$ and $(\beta_k)_{k\in \NN}$ are bounded then $(J_i(x^k,\beta_k))_{k\in \NN}$ and $(\bar{x}_i^k)_{k\in \NN}$ are bounded. This implies the boundedness of $(\|A_i(\bar{x}_i^k)+\bar{u}_i^k\|)_{k\in \NN}$ for all $i\in \mathbb{I}$.

 Using Proposition \ref{prop21}(iii), the right side of \eqref{pasar al lim1} goes to 0 when $k$ goes to $\infty$, establishing the result.
\end{proof}

 Next we establish our main convergence result for {\bf Algorithm 1}.
\begin{theorem}\label{teo1}
The sequence $(x^k)_{k\in \NN}$ converges to some element belonging to $S_*\cap X $.
 \end{theorem}
\begin{proof}
 We claim that there exists a cluster point of $(x^k)_{k\in \NN}$ belonging to $S_*$. The existence of the cluster points follows from Proposition \ref{prop21}(ii). Let $(x^{j_k})_{k\in \NN}$ be a convergent subsequence of $(x^k)_{k\in \NN}$ such that, for all $i\in \mathbb{I}$ the sequences  $(\bar{x}_i^{j_k})_{k\in \NN}, (\bar{u}_i^{j_k})_{k\in \NN}, (\alpha_{j_k,i})_{k\in \NN}$ and $(\beta_{j_k})_{k\in \NN}$ are convergent, and $\lim _{k\to \infty}x^{j_k}= \tilde{x}$.\\
Using Proposition \ref{prop21}(iii) and taking limits in \eqref{desig-muy-usada} over the subsequence $(j_k)_{k\in \NN}$, we have for all $i\in \mathbb{I}$,
\begin{equation}\label{limite1}
0=\lim_{k\to \infty}\la A_i(\bar{x}_i^{j_k})+\bar{u}_i^{j_k},x^{j_k}-\bar{x}_i^{j_k} \ra \ge \lim_{k\to \infty}  \frac{\alpha_{j_k,i}\delta}{\hat{\beta}}\|x^{j_k}-J_i(x^{j_k},\beta_{j_k})\|^2\geq0.
\end{equation}
Therefore,
\begin{equation*}
\lim_{k\to \infty} \alpha_{j_k,i}\|x^{j_k}-J_i(x^{j_k},\beta_{j_k})\|=0.
\end{equation*}
Now consider the two possible cases.

(a) First, assume that $\lim_{k\to \infty}\alpha_{j_k,i}\ne 0$, i.e., $\alpha_{j_k,i}\ge \bar{\alpha}$ for all $k$ and some  $\bar{\alpha}>0$. In view of (\ref{limite1}),

    \begin{equation}\label{limcero1}
    \lim_{k\to \infty}\|x^{j_k}-J_i(x^{j_k},\beta_{j_k})\|=0.
    \end{equation}
    Since $J_i$ is continuous, by continuity  of $A_i$ and $(I+\beta_k B_i)^{-1}$ and by Proposition \ref{inversa}, \eqref{limcero1} becomes
    \begin{equation*}
    \tilde{x}=J_i(\tilde{x},\tilde{\beta}),
    \end{equation*}
    which implies that $\tilde{x}\in S_i^*$ for all $i\in \mathbb{I}$. Then $\tilde{x}\in S_*$ establishing the claim.

 (b) On the other hand, if $\lim_{k\to \infty}\alpha_{j_k,i}=0$ then for $\theta \in (0,1)$ as in the {\bf Algorithm 1}, we have
    $$\lim_{k\to\infty}\frac{\alpha_{j_k,i}}{\theta}=0.$$
    Define
    $$y^{j_k}_i:=\frac{\alpha_{j_k,i}}{\theta}J_i(x^{j_k},\beta_{j_k})+\Big(1-\frac{\alpha_{j_k,i}}{\theta}\Big)x^{j_k}.$$
    Then,
    \begin{equation}\label{ykgox1}
    \lim_{k\to\infty}y_i^{j_k}=\tilde{x}.
    \end{equation}
    Using the definition of the $j_i(k)$ and \eqref{alphak}, we have that $y_i^{j_k}$ does not satisfy \eqref{jk1} implying
    \begin{equation*}
    \Big \la A_i (y^{j_k}_i)+u^{j_k}_{j_i(k)-1}-\frac{\delta}{\beta_k}(x^k -J_i\big(x^k,\beta_k)\big), x^k-J_i(x^k,\beta_k)\Big \ra > 0,
    \end{equation*}
    equivalent to
    \begin{equation}\label{conse1}
\Big \la A_i (y^{j_k}_i)+u^{j_k}_{j(j_k)-1,i}, x^k-J_i(x^k,\beta_k)\Big \ra > \frac{\delta}{\beta_k}\|x^k-J_i\big(x^k,\beta_k\big)\|^2,
    \end{equation}
    for $u^{j_k}_{j(j_k)-1,i}\in B_i(y^{j_k}_i)$ and all $k\in \NN$ and $i\in \mathbb{I}$.\\
    Redefining the subsequence $(j_k)_{k\in \NN}$, if necessary, we may assume that  $(u^{j_k}_{j(j_k)-1,i})_{k\in \NN}$ converges to $\tilde{u}_i$. By the maximality of
$B_i$, $\tilde{u}_i$ belongs to $B_i(\tilde{x})$. Using the continuity of $J_i$, $(J(x^{j_k},\beta_{j_k}))_{k\in \NN}$ converges to $J_i(\tilde{x},\tilde{\beta})$. Using
\eqref{ykgox1} and taking limit in (\ref{conse1}) over the subsequence $(j_k)_{k\in \NN}$ we have
    \begin{equation}\label{tiu1}
    \Big\la A_i(\tilde{x}) + \tilde{u}_i, \tilde{x}- J_i(\tilde{x},\tilde{\beta})\Big\ra \le \frac{\delta}{\tilde{\beta}}\|\tilde{x}-J_i(\tilde{x},\tilde{\beta})\|^2.
    \end{equation}
    Using (\ref{jota1}) and multiplying by $\tilde{\beta}$ on both sides of (\ref{tiu1}) we get
    \begin{equation*}
    \la \tilde{x}-J_i(\tilde{x},\tilde{\beta})-\tilde{\beta}\tilde{v}_i+\tilde{\beta}\tilde{u}_i, \tilde{x}-J_i(\tilde{x},\tilde{\beta})\ra \le \delta\| \tilde{x}-J_i(\tilde{x},\tilde{\beta})\|^2,
    \end{equation*}
    where $\tilde{v}_i\in B_i(J_i(\tilde{x},\tilde\beta))$. Applying the monotonicity of $B_i$, we obtain
    $$\| \tilde{x}-J_i(\tilde{x},\tilde{\beta})\|^2 \le \delta \| \tilde{x}-J_i(\tilde{x},\tilde{\beta})\|^2,$$
    implying that $\| \tilde{x}-J_i(\tilde{x},\tilde{\beta})\|\le 0$. Thus, $\tilde{x}=J_i(\tilde{x},\tilde{\beta})$ and hence, $\tilde{x}\in S_*^i$ for all $i\in \mathbb{I}$, thus $\tilde{x} \in S_*$.

\end{proof}

\subsection{Convergence Analysis of Algorithm \ref{concep1}}
In this subsection we prove the convergence analysis of {\bf Algorithm \ref{concep1}}.
\begin{proposition}\label{prop2}
\begin{enumerate}
\item The sequence $(x^k)_{k\in \NN}$ is Fej\'er convergent to $S_*\cap X$.
\item The sequence $(x^k)_{k\in \NN}$ is bounded.
\item $\lim_{k\to \infty}\|x^{x+1}-x^k\|=0$.
\end{enumerate}
\end{proposition}
\begin{proof}
\begin{enumerate}
\item Take $x^*\in S_*\cap X$.  Using \eqref{Fk}, Proposition \ref{proj}(i) and Lemma \ref{propseq}, we have
\begin{eqnarray}\label{fejer-des}\nonumber\|x^{k+1}-x^{*}\|^2&=& \|P_{X}(P_{H_{\rho(k)}(\bar{x}^k,\bar{u}^k)}(x^k))-P_X(P_{H_{\rho(k)}(\bar{x}^k,\bar{u}^k)}(x^{*}))\|^2\\\nonumber &\leq& \|P_{H_{\rho(k)}(\bar{x}^k,\bar{u}^k)}(x^k))-P_{H_{\rho(k)}(\bar{x}^k,\bar{u}^k)}(x^{*}) \|^2 \\ &\leq& \|x^k-x^*\|^2-\|P_{H_{\rho(k)}(\bar{x}^k,\bar{u}^k)}(x^k)-x^k\|^2 .
\end{eqnarray}
So, $\|x^{k+1}-x^{*}\|\leq \|x^k-x^*\|$.

\item Follows immediately from item (i) and Proposition \ref{punto}(i).

\item Using that the sequence $(x^k)_{k\in\NN}\subset X$ by \eqref{Fk}, we have
$$ \|x^{k+1}-x^k\|^2=\|P_X\big(P_{H_{\rho(k)}(\bar{x}^k,\bar{u}^k)}(x^k)\big)-x^k\|^2\leq \|P_{H_{\rho(k)}(\bar{x}^k,\bar{u}^k)}(x^k)-x^k\|^2.$$
Now, using \eqref{fejer-des} we obtain,
\begin{equation}\label{zero12}\|x^{k+1}-x^k\|^2\leq\|x^k-x^*\|^2-\|x^{k+1}-x^*\|^2 .\end{equation}
Taking limit in \eqref{zero12}, using Proposition \ref{punto}(ii) and item (i), we get the result.
\end{enumerate}
\end{proof}
\begin{proposition}\label{cadai}
For the sequences $(\bar{u}^k)_{k\in\NN}, \ (\bar{x}^k)_{k\in\NN}$ and  $(x^k)_{k\in\NN}$ generated by {\bf Algorithm 2} holds
$$\lim_{k\to \infty}\la A_{\rho(k)}(\bar{x}^k)+\bar{u}^k,x^k-\bar{x}^k \ra=0.$$
\end{proposition}
\begin{proof}

Reordering \eqref{fejer-des},  for all $x^*\in S_*\cap X$ we get
$$\|P_{H_{\rho(k)}(\bar{x}^k,\bar{u}^k)}(x^k)-x^k\|^2 \leq |x^k-x^*\|^2-\|x^{k+1}-x^*\|^2.$$
Using the fact that
$$P_{H_{\rho(k)}(\bar{x}^k,\bar{u}^k)}(x^k)=x^k-\frac{\la A_{\rho(k)}(\bar{x}^k)+\bar{u}^k,x^k-\bar{x}^k  \ra}{\|A_{\rho(k)}(\bar{x}^k)+\bar{u}^k\|^2}(A_{\rho(k)}(\bar{x}^k)+\bar{u}^k),$$
and the previous equation, we have
\begin{equation}\label{pasar al lim}
\frac{(\la A_{\rho(k)}(\bar{x}^k)+\bar{u}^k,x^k-\bar{x}^k  \ra)^2}{\|A_{\rho(k)}(\bar{x}^k)+\bar{u}^k\|^2}\leq \|x^k-x^*\|^2-\|x^{k+1}-x^*\|^2.
\end{equation}
By Proposition \ref{inversa} and the continuity of $A_{\rho(k)}$ we have that $J_{\rho(k)}$ is continuo. Since $(x^k)_{k\in \NN}$ and $(\beta_k)_{k\in \NN}$ are bounded we get that $(\bar{x}^k)_{k\in \NN}$ is bounded, implying the boundedness of $(\|A_{\rho(k)}(\bar{x}^k)+\bar{u}^k\|)_{k\in \NN}$.
 As we had seen by Proposition \ref{zero12} and Proposition \ref{prop2}(i), the right side of \eqref{pasar al lim} goes to 0 when $k$ goes to $\infty$, establishing the result.
\end{proof}

 Next we establish our main convergence result for {\bf Algorithm 2}.
\begin{theorem}\label{teo1}
The sequence $(x^k)_{k\in \NN}$ converges to some element belonging to $S_*\cap X $.
 \end{theorem}
\begin{proof}
 Since  $(x^k)_{k\in \NN}$ is bounded then have cluster points, as every $x^k$ belong to $X$ by \eqref{Fk} and $X$ is closed, then all clusters point of $(x^k)_{k\in\NN}$ belong to $X$.

 It is sufficient to proof that at least one cluster point of the sequence $(x^k)_{k\in\NN}$ belongs to the solution set. Applying Proposition \ref{punto}(iii) and Proposition \ref{prop2}(i) we will get that the whole sequence is convergent to the solution set. Note that due to the property of periodicity and surjective (as was defined) of the function $\rho:\NN\rightarrow \mathbb{I}$, there exist subsequences $(n^i_k)_{k\in\NN, i \in \mathbb{I}}$ such that $|n^i_k-n^j_k| \leq m$  for all $k\in\NN$ and all $i,j\in \mathbb{I}$ and $\rho(n^i_k)=i$ for all $i\in\mathbb{I}$. Using Proposition \ref{prop2}(iii) it is easy to proof that $\lim_{k\rightarrow \infty}\|x^{n^i_k}-x^{n^j_k}\|^2 = 0$, for all $i,j\in \mathbb{I}$. This result imply that all subsequences $(x^{n^i_k})_{k\in\NN}$ for $ i\in\mathbb{I}$, have the same cluster points.
Now we can take subsequences of $(x^{n^i_{l_k}})$  such that $x^{n^i_{l_k}} \rightarrow \tilde{x}$ for all $i \in \mathbb{I}$. This subsequence can be chosen such that $(\alpha_{n_{l_k}})_{k\in\NN}$ and $(\beta_{n_{l_k}})_{k\in \NN}$ be convergent and let $\lim_{k\to\infty}\beta_{{n_{l_k}}}=\tilde{\beta}$.

Rewriting \eqref{desig-muy-usada} for the case of {\bf Algorithm 2}, we have
\begin{equation}\label{usada}
\la A_{\rho(k)}(\bar{x}^{k})+\bar{u}^{k},x^{k}-\bar{x}^{k} \ra \ge \frac{\alpha_{k}\delta}{\hat{\beta}}\|x^{k}-J_{\rho(k)}(x^k,\beta_k)\|^2.
\end{equation}
Using Proposition \ref{cadai} and taking limits in \eqref{usada} over the subsequences $(n^i_{l_k})_{k\in \NN}$, we have for all $i\in \mathbb{I}$,

\begin{equation}\label{limite}
0=\lim_{k\to \infty}\la A_i(\bar{x}^{n_{l_k}})+\bar{u}^{n_{l_k}},x^{n_{l_k}}-\bar{x}^{n_{l_k}} \ra \ge \lim_{k\to \infty}  \frac{\alpha_{n_{l_k}}\delta}{\hat{\beta}}\|x^{n_{l_k}}-J_i(x^{n_{l_k}},\beta_{n_{l_k}})\|^2\geq 0.
\end{equation}
Therefore,
\begin{equation*}
\lim_{k\to \infty}{\alpha_{n_{l_k}}\|x^{n_{l_k}}-J_i(x^{n_{l_k}}},\beta_{n_{l_k}})\|^2=0.
\end{equation*}
Now consider the two possible cases.

(a) First, assume that $\lim_{k\to \infty}\alpha_{n_{l_k}}\ne 0$, i.e., $\alpha_{n_{l_k}}\ge \bar{\alpha}$ for all $k$ and some  $\bar{\alpha}>0$. In view of (\ref{limite}),

    \begin{equation}\label{limcero}
    \lim_{k\to \infty} \|x^{n_{l_k}}-J_i(x^{n_{l_k}},\beta_{n_{l_k}})\| =0.
    \end{equation}
    Since $J_i$ is continuous, by continuity of $A_i$ and $(I+\beta_k B_i)^{-1}$  by Proposition \ref{inversa}, \eqref{limcero} becomes
    \begin{equation*}
    \tilde{x}=J_i(\tilde{x},\tilde{\beta})=(I+\tilde{\beta} B_i)^{-1}(I-\tilde{\beta} A_i)(\tilde{x}),
    \end{equation*}
    which implies that $\tilde{x}\in S_*^i$ for all $i\in \mathbb{I}$ using Proposition \ref{parada}. Then $\tilde{x}\in S_*$ establishing the claim.

 (b) On the other hand, if $\lim_{k\to \infty}\alpha_{_{l_k}}=0$ then for $\theta \in (0,1)$ as in {\bf Algorithm 2}, we have
    $$\lim_{k\to\infty}\frac{\alpha_{n_{l_k}}}{\theta}=0.$$
    Define
    $$y^{n_{l_k}}_i:=\frac{\alpha_{n_{l_k}}}{\theta}J_i(x^{n_{l_k}},\beta_{n_{l_k}})+\Big(1-\frac{\alpha_{n_{l_k}}}{\theta}\Big)x^{n_{l_k}},$$
    then,
    \begin{equation}\label{ykgox}
    \lim_{k\to\infty}y_i^{n_{l_k}}=\tilde{x}.
    \end{equation}
    Using the definition of the $j(k)$ and \eqref{alphak}, we have that $y_i^{j_k}$ does not satisfy \eqref{jk} implying
    \begin{equation}\label{conse}
    \Big \la A_i (y^{n_{l_k}}_i)+u^{n_{l_k}}_{j(n_{l_k})-1,i}, x^{n_{l_k}}-J_i(x^{n_{l_k}},\beta_{n_{l_k}})\Big \ra <\frac{\delta}{\beta_{n_{l_k}}}\|x^{n_{l_k}} -J_i(x^{n_{l_k}},\beta_{n_{l_k}})\|^2,
    \end{equation}

    for $u^{n_{l_k}}_{j(n_{l_k})-1,i}\in B_i(y^{n_{l_k}}_i)$ and all $k\in \NN$ and $i\in \mathbb{I}$.\\
    Redefining the subsequence $(n_{l_k})_{k\in \NN}$, if necessary, we can assume that  $(u^{n_{l_k}}_{j(n_{l_k})-1,i})_{k\in \NN}$ converges to $\tilde{u}_i$, due to the maximality of $B_i$ we obtain that $\tilde{u}_i$ belongs to $B_i(\tilde{x})$. Using the continuity of $J_i$,  \eqref{ykgox} and taking limit in (\ref{conse}) over the subsequence $(n_{l_k})_{k\in \NN}$ we have
    \begin{equation}\label{tiu}
    \Big\la A_i(\tilde{x}) + \tilde{u}_i, \tilde{x}- J_i(\tilde{x},\tilde{\beta})\Big\ra \le \frac{\delta}{\tilde{\beta}}\|\tilde{x}-J_i(\tilde{x},\tilde{\beta})\|^2.
    \end{equation}
    Using the definition of $J_i(\tilde{x},\tilde{\beta}):=(I+\tilde{\beta} B_i)^{-1}(I-\tilde{\beta} A_i)(\tilde{x})$ and multiplying by $\tilde{\beta}$ on both sides of (\ref{tiu}), we get
    \begin{equation*}
    \la \tilde{x}-J_i(\tilde{x},\tilde{\beta})-\tilde{\beta}\tilde{v}_i+\tilde{\beta}\tilde{u}_i, \tilde{x}-J_i(\tilde{x},\tilde{\beta})\ra \le \delta\| \tilde{x}-J_i(\tilde{x},\tilde{\beta})\|^2,
    \end{equation*}
    where $\tilde{v}_i\in B_i(J_i(\tilde{x},\tilde\beta))$. Applying the monotonicity of $B_i$, we obtain
    $$\| \tilde{x}-J_i(\tilde{x},\tilde{\beta})\|^2 \le \delta \| \tilde{x}-J_i(\tilde{x},\tilde{\beta})\|^2,$$
    implying that $\| \tilde{x}-J_i(\tilde{x},\tilde{\beta})\|\le 0$. Thus, $\tilde{x}=J_i(\tilde{x},\tilde{\beta})$ and hence, $\tilde{x}\in S_*^i$ for all $i\in \mathbb{I}$, thus $\tilde{x} \in S_*$ This prove the convergence of the whole sequence to a point of the set $S_*\cap X$.
\end{proof}

\section{Numerical Experiments}\label{examples}

 In this section, we compare numerically {\bf Algorithm 1} and {\bf Algorithm 2} through two examples. In one of the examples, we compare both algorithms with Algorithm 3.3 in \cite{van}. We use MATLAB version R2015B on a PC with Intel(R) Core(TM) i5-4570 CPU 3.20GHz and Windows 7 Enterprise, Service Pack 1. For the calculation of the projection steps we use the Quadratic Programming (quadprog) tool. In both examples, we consider the system of inclusion problems with the operators $B_i=N_C$ for $i=1,2,\cdots,m$, $C=\{x\in\RR^n: Ax\leq b\}$ where $A\in M^{l,n}(\RR)$ (a matrix with $l$ rows and $n$ columns with real entries) and $b\in \RR^l_+$ are computed randomly. In both cases, we take $l=20$, $\delta=0.1$, $\theta=0.5$, $\beta_k=1$ for all $k\in\NN$ and the initial point is $x^0=(1,1,\cdots,1)^T\in \RR^n$. The tolerance is taken as $\|x^k-x^*\|\leq 0.001$, with $x^*$ being the solution. Note that in both examples the problem turns into a system of variational inequalities, then we can choose $X=C$, no  need to be bounded. In {\bf Table 1} and {\bf Table 2} we denote the number of iterations by {\it iter} and the number of evaluation of the operators by {\it nT}.

\begin{example}\cite{harp,van}\label{exa1}
For all $i=1,2,\cdots, m$ consider the operator $T_i=M_i x$, with
\begin{equation*}
M_i=Q_i ^{T}Q_i,
\end{equation*}
where $Q$ is an $n\times n$ random matrix, then the matrix $M_i$ is positive semi-definite, hence the operator $T_i$ is maximal monotone. We compare our two algorithms with Algorithm 3.3 in \cite{van}. We took the same initial data used in \cite{van}. See results in {\bf Table 1}.

\begin{tabular}{|p{1cm} p{1cm}| p{1.8cm} p{1.8cm}| p{1.8cm} p{1.8cm}| p{1.8cm} p{1.6cm}|}
 \hline
 \multicolumn{8}{|c|}{{\bf Table 1} Results for Example \ref{exa1}.}\\
 \hline
  & &  {\bf Algor 1}  & & {\bf Algor 2} & & {\bf Algor 3.3} &  in\cite{van}\\
 \hline
 $n$ &$m$ &  iter(nT)   & CPU time  &  iter(nT) & CPU time & iter(nT)   & CPU time \\
 \hline
 2& 10 & 7(132) & 0.561604 & 30(53) & 0.234001 & 83(853) & 4.14963 \\
 5& 10 & 22(774) & 1.40401 & 100(387) & 0.670804 &  150(1887) & 10.7797\\
 10& 10 &46(2377)& 2.04361 & 110(576)& 0.702004 &  584(6416) & 190.181\\
 2& 20 & 6(224) & 0.780005 & 40(73) & 0.234002 &  57(1213) & 4.50843 \\
 5& 20 & 19(1300) & 1.76281 & 100(401) & 0.686404 &  232(5041) & 65.692\\
 10 & 20 & 29(3039) & 2.38682 & 120(640)& 0.811205 &  261(5860) & 88.7802 \\
 20 & 30 & 33(7308) & 4.49283 & 150(1111)& 1.21681 &  --(--) & $>10^4$ \\
 30 & 30 & 67(18032) & 9.39126 & 210(1741)& 1.85641 & --(--) & $>10^4$ \\
 50 & 30 & 87(27381) & 14.7265 & 420(4408)& 4.58643 & --(--) & $>10^4$ \\
 \hline
\end{tabular}
\end{example}

\begin{example}\label{exa2}
Consider the operators $A_i=M_i x+f(x)$, where $M_i$ is obtained as in {\bf Example \ref{exa1}} and $f:\RR^n\rightarrow \RR^n$ is defined for each $x=(x_1,x_2,\cdots,x_n)$ by
\begin{equation*}
f(x)=(x_1^3,x_2^3,\cdots,x_n^3).
\end{equation*}
So, the operators $A_i$ are maximal monotone and continuous but non-Lipschitz continuous. We don't compare our algorithms with others algorithms in the literature. As well as we know, there not algorithms for systems of variational inequalities where the operators being non-Lipschitz. See results in {\bf Table 2}.

\begin{tabular}{|p{1cm} p{1cm}| p{2.5cm} p{2.8cm}| p{2.5cm} p{2.8cm}|}
 \hline
 \multicolumn{6}{|c|}{{\bf Table 2} Results for Example \ref{exa2}.}\\
 \hline
  & &  {\bf Algorithm 1}  & & {\bf Algorithm 2} &\\
 \hline
 $n$ &$m$ &  iter(nT)   & CPU time  &  iter(nT) & CPU time  \\
 \hline
 5& 10 & 21(913) & 1.09201 & 80(375) & 0.499203 \\
 20& 10 & 71(5941) & 3.57242 & 260(2183) & 2.02801 \\
 50& 10 &  127(14128) & 7.84685 & 450(5023) & 4.38363 \\
 5& 20 & 15(1286) & 1.35721 & 140(716) & 0.982806 \\
 20& 20 & 39(6531) & 3.88442 & 160(1327) & 1.29481 \\
 50& 20 & 63(14676) & 7.47245 & 380(4282)& 4.04043 \\
 \hline
\end{tabular}

\end{example}
\begin{remark}
The numerical results confirm that ours proposed algorithms have a competitive behavior respect to similar methods, such as Algorithm 3.3 in \cite{van}.  Note that Algorithm 3.3 in \cite{van} have the better results, compared with others algorithm proposed in \cite{van}, as can be seen in its numerical experiments.  The advantage of {\bf Algorithm 1} and {\bf Algorithm 2} over Algorithm 3.3 lies in the difference between the number of iterations and the CPU time.
\end{remark}

\section{Conclusions}
We present two algorithms for solving systems of inclusion problems for the sum of two maximal monotone operators in Euclidean spaces with finite dimension. Both algorithms are variants of the forward-backward splitting method. One of them is also a hybrid with the alternating projection method. The algorithms contain two steps, a line-search, and a projection onto the separating hyperplane. The convergence analysis of both algorithms is established assuming maximal monotonicity of all the operators without the hypothesis of Lipschitz continuity. The numerical experiments show a better performance for our algorithms when compared with similar ones in the literature. The analysis of the complexity and the development of these algorithms for Banach spaces is a topic for future research.

\section*{Acknowledgments}
The author was partially supported by CNPq grant 200427/2015-6. This
work was concluded while the author was visiting the School of Information Technology and Mathematical Sciences at the  University of South
Australia. The author would like to thank the great hospitality received during his visit, particularly to Regina S. Burachik and  C. Yal\c{c}in Kaya. The author would like to express his gratitude to two anonymous referees for their valuable comments and suggestions that are very helpful to improve this paper.
\bigskip

\bibliographystyle{plain}

\end{document}